\documentclass[10pt,a4paper]{article}

\usepackage{tikz}
\usetikzlibrary{shapes,arrows,positioning,fit,calc}
\usepackage{amsmath,amssymb}

\usepackage[utf8]{inputenc}            
\usepackage[a4paper,margin=3cm]{geometry}

\usepackage{amsmath,amssymb,amsfonts,amsthm,mathtools}

\usepackage[algo2e,linesnumbered,ruled,vlined]{algorithm2e}

\usepackage{graphicx}
\usepackage{float}
\usepackage{subfig}
\usepackage{xcolor}                    

\usepackage{ifthen}
\usepackage{enumerate}
\usepackage{cuted}
\usepackage{subeqnarray}
\usepackage{url}
\usepackage{todonotes}

\usepackage{hyperref}

\newcommand{\R}{\mathbb{R}}
\newcommand{\N}{\mathbb{N}}

\newcommand{\inner}[2]{%
  \ifthenelse{\equal{#2}{}}%
    {\langle\cdot,\cdot\rangle_{#1}}%
    {\langle#2\rangle_{#1}}}

\newcommand{\norm}[2]{%
  \ifthenelse{\equal{#2}{}}%
    {\lVert\cdot\rVert_{#1}}%
    {\lVert#2\rVert_{#1}}}

\newcommand{\seminorm}[2]{%
  \ifthenelse{\equal{#2}{}}%
    {\lvert\cdot\rvert_{#1}}%
    {\lvert#2\rvert_{#1}}}

\newcommand{\cC}{{\cal C}}

\newcommand{\cF}{{\cal F}}

\newcommand{\cH}{{\cal H}}

\newcommand{\cK}{{\cal K}}

\newcommand{\cP}{{\cal P}}

\newcommand{\cT}{{\cal T}}

\newcommand{\cV}{{\cal V}}

\newcommand{\cX}{{\cal X}}

\newcommand{\cZ}{{\cal Z}}

\newcommand{\fS}{{\mathfrak S}}
\newcommand{\co}{\operatorname{co}}

\newcommand{\bC}{\mathbf{C}}

\DeclareMathOperator*{\argmin}{arg\,min}

\newtheorem{theorem}{Theorem}
\newtheorem{definition}{Definition}

\newtheorem{remark}{Remark}

\newcommand{\calx}{\mathcal{X}}
 
\newcommand{\RR}{\mathbb{R}}

\newcommand{\mN}{{\mathbb N}}

\newcommand{\mC}{{\mathbb C}}

\newcommand{\mR}{{\mathbb R}}

\newcommand{\mU}{{\mathbb U}}

\newcommand{\gs}[1]{{\color{red}GS: #1}}

\DeclareUnicodeCharacter{202C}{\textvisiblespace}

\graphicspath{{./Figures/}}
\usepackage{authblk} 

\title{Kernel Methods for the Construction of Certified Lyapunov Functions via Approximate Koopman Eigenfunctions}

\author[1]{P.~Giesl\thanks{p.a.giesl@sussex.ac.uk}}
\author[2]{S.~Hafstein\thanks{shafstein@hi.is}}
\author[3,4]{B.~Hamzi\thanks{boumediene.hamzi@gmail.com}}
\author[4]{J.~Lee\thanks{jlee9@caltech.edu}}
\author[4]{H.~Owhadi\thanks{owhadi@caltech.edu}}
\author[5]{G.~Santin\thanks{gabriele.santin@unive.it}}
\author[6]{U.~Vaidya\thanks{uvaidya@clemson.edu}}

\affil[1]{Department of Mathematics, University of Sussex,  UK}
\affil[2]{Faculty of Physical Sciences, University of Iceland, Iceland}
\affil[3]{Department of Computing and Mathematical Sciences, Caltech, USA}
\affil[4]{Alan Turing Institute, London, UK}
\affil[5]{Department of Environmental Sciences, Informatics and Statistics, Ca' Foscari University of Venice, Italy}
\affil[6]{Department of Mechanical Engineering, Clemson University, USA}

\date{\today}

\begin{document}

\maketitle

\begin{abstract}
We present a kernel-based methodology for constructing Lyapunov functions 
for nonlinear dynamical systems using approximate Koopman eigenfunctions. 
Our approach decomposes principal Koopman eigenfunctions into linear and 
nonlinear components, where the linear part is obtained from the system's 
linearization and the nonlinear part is computed by solving a partial 
differential equation using symmetric kernel collocation in reproducing 
kernel Hilbert spaces (RKHS). The resulting Lyapunov function is constructed 
as a quadratic form in the approximate eigenfunctions. We establish error 
bounds relating the approximation quality to the fill distance of collocation 
points and provide a certification procedure using continuous piecewise 
affine (CPA) methods. Numerical experiments on benchmark systems, including 
a polynomial system and the Duffing oscillator, demonstrate the effectiveness 
of our approach.

\end{abstract}

\section{Introduction}

Time series data permeate numerous scientific domains, spurring developments in statistical and machine learning forecasting techniques \cite{kantz97,CASDAGLI1989, yk1, yk2, yk3, yk4, survey_kf_ann,jaideep1,nielsen2019practical,abarbanel2012analysis, pillonetto2011new,wang2011predicting,brunton2016discovering,lusch2018deep,callaham2021learning,kaptanoglu2021physicsconstrained,kutz2022parsimony}. Dynamical systems theory equips us with tools to analyze the underlying dynamics governing such data.

Lyapunov functions play a pivotal role in the qualitative theory of dynamical systems, providing certificates of stability without requiring explicit solution of the underlying differential equations. They are primarily used to define basins of attraction for asymptotically stable equilibria and to facilitate control design. However, obtaining an explicit analytical expression for a Lyapunov function associated with a nonlinear differential equation is typically infeasible, motivating the development of computational methods. These include the sums of squares (SOS) approach for polynomial Lyapunov functions via semidefinite optimization \cite{prajna}, the construction of continuous piecewise affine (CPA) Lyapunov functions using linear optimization \cite{Haf2007mon}, methods based on Zubov's equation \cite{zubov2001camilli}, set-oriented approaches \cite{book2002grune}, and radial basis function methods \cite{Giesl2007}, which have also been extended to data-driven frameworks \cite{lyap_bh}.

In parallel, the foundational work by Koopman \cite{koopman1932dynamical} introduced Koopman operator theory, which provides a linear perspective on nonlinear dynamical systems by mapping the finite-dimensional state evolution to an infinite-dimensional, linear functional framework. This approach gained traction following key advancements \cite{mezic2005spectral}, highlighting the Koopman operator's spectral richness and its critical role in various analytical and synthesis challenges. Principal eigenfunctions of the Koopman operator illuminate state-space geometry and characterize stable and unstable manifolds of equilibrium points \cite{mauroy2016global,umathe2023spectral}. Studies such as \cite{matavalam2024data} have utilized these eigenfunctions to identify stability boundaries and domains of attraction. Moreover, Koopman eigenfunctions facilitate constructing solutions to optimal control problems \cite{vaidya2022spectral,vaidya2025koopman} and have led to novel computational approaches like the path-integral formula \cite{deka2023path}. Despite its utility, the Koopman operator often presents computational challenges due to its potentially continuous spectrum and the associated issue of ``spectral pollution'' when approximated by finite-dimensional matrices \cite{colbrook2020foundations,colbrook2024limits}.

Capitalizing on the theory of reproducing kernel Hilbert spaces (RKHS) \cite{CuckerandSmale}, kernel-based methods enhance regularization, convergence, and interpretability while solidifying the mathematical foundation for analyzing dynamical systems \cite{chen2021solving, houman_cgc, yk1, bhcm11, bhcm1, lyap_bh, BHPhysicaD, hamzi2019kernel, bh2020b, klus2020data, ALEXANDER2020132520, bhks, bh12, bh17, hb17, mmd_kernels_bh, bh_kfs_p5, bh_kfs_p6, hou2024propagating, cole_hopf_poincare, kernel_sos} and surrogate modeling \cite{santinhaasdonk19}. Our previous work \cite{Lee2025} addressed the spectral pollution challenge by developing a kernel-based methodology to directly construct principal eigenfunctions from data, bypassing explicit operator approximation. Unlike traditional approaches that first construct a Koopman operator and then derive its spectra, our method directly learns eigenfunctions by solving the associated partial differential equations in RKHS, following the path-integral formulation of \cite{deka2023path}. This approach decomposes principal eigenfunctions into linear and nonlinear components, integrating insights from system linearization with the approximation power of RKHSs.

This paper extends our kernel-based eigenfunction methodology to construct Lyapunov functions. The key insight, following \cite{Lee2025b}, is that Lyapunov functions can be expressed as quadratic forms in Koopman eigenfunctions. We decompose each principal eigenfunction into linear and nonlinear parts: the linear part is determined by the system's linearization at the equilibrium, while the nonlinear part satisfies a first-order PDE that we solve via symmetric kernel collocation. The resulting approximate Lyapunov function inherits rigorous error bounds (Theorem~4) that depend on the fill distance of collocation points. We further provide a certification procedure using CPA methods that verifies the computed function is a valid Lyapunov function outside the training points. The overall procedure is summarized in Figure \ref{fig:methodology_flowchart}.

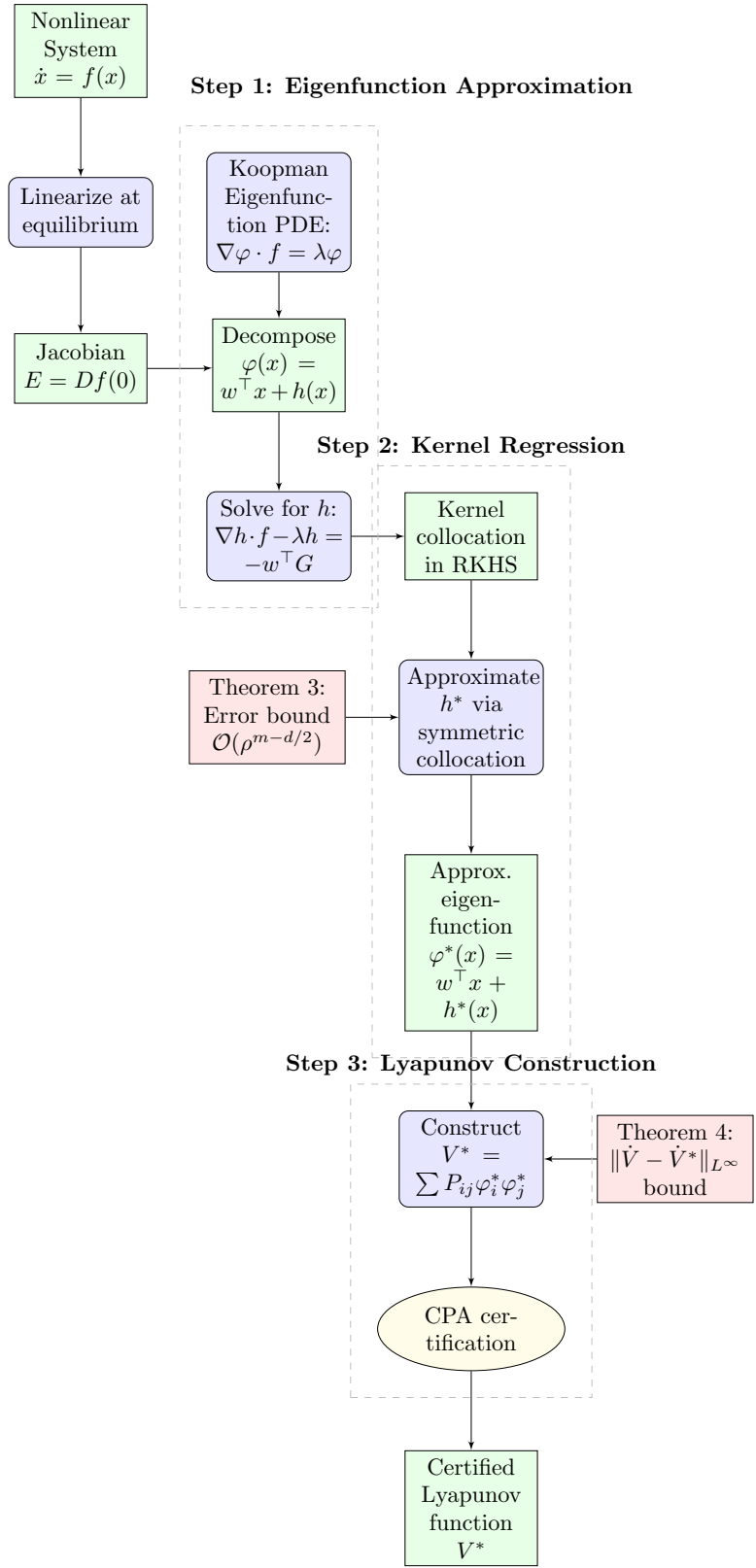
\begin{figure}[htbp]
\centering
\scalebox{0.9}{
\begin{tikzpicture}[
    node distance=1.2cm and 0.8cm,
    block/.style={rectangle, draw, fill=blue!10, text width=5.5em, text centered, rounded corners, minimum height=3em},
    data/.style={rectangle, draw, fill=green!10, text width=5em, text centered, minimum height=3em},
    op/.style={circle, draw, fill=orange!20, text width=2em, text centered, minimum height=2em},
    theorem/.style={rectangle, draw, fill=red!10, text width=6em, text centered, minimum height=3em},
    line/.style={draw, -latex'},
    cloud/.style={ellipse, draw, fill=yellow!10, text width=5em, text centered, minimum height=3em},
    phase/.style={draw=gray!50, dashed, inner sep=0.4cm} 
]

\node[data] (input) {Nonlinear System $\dot{x}=f(x)$};
\node[block, below=of input] (linearize) {Linearize at equilibrium};
\node[data, below=of linearize, yshift=-0.1cm] (E) {Jacobian $E=Df(0)$};
\node[block, right=of linearize] (koopman) {Koopman Eigenfunction PDE: $\nabla\varphi\cdot f=\lambda\varphi$};
\node[data, below=of koopman, yshift=0.5cm] (decomp) {Decompose $\varphi(x) = w^\top x + h(x)$};
\node[block, below=of decomp] (pde) {Solve for $h$: $\nabla h\cdot f - \lambda h = -w^\top G$};
\node[data, right=of pde] (kernel) {Kernel collocation in RKHS};
\node[block, below=of kernel] (approx) {Approximate $h^*$ via symmetric collocation};
\node[theorem, left=of approx] (thm3) {Theorem 3: Error bound $\mathcal{O}(\rho^{m-d/2})$};
\node[data, below=of approx] (eigen) {Approx. eigenfunction $\varphi^*(x) = w^\top x + h^*(x)$};
\node[block, below=of eigen] (lyap) {Construct $V^* = \sum P_{ij}\varphi_i^*\varphi_j^*$};
\node[theorem, right=of lyap] (thm4) {Theorem 4: $\|\dot{V}-\dot{V}^*\|_{L^\infty}$ bound};
\node[cloud, below=of lyap] (cpa) {CPA certification};
\node[data, below=of cpa] (output) {Certified Lyapunov function $V^*$};

\path[line] (input) -- (linearize);
\path[line] (linearize) -- (E);
\path[line] (E) -- ++(1.8,0) -- (decomp); 
\path[line] (koopman) -- (decomp);
\path[line] (decomp) -- (pde);
\path[line] (pde) -- (kernel);
\path[line] (kernel) -- (approx);
\path[line] (approx) -- (eigen);
\path[line] (eigen) -- (lyap);
\path[line] (thm3) -- (approx);
\path[line] (lyap) -- (cpa);
\path[line] (thm4) -- (lyap);
\path[line] (cpa) -- (output);


\node[phase, fit=(koopman)(decomp)(pde)] (phase1) {};
\node[anchor=south, yshift=0.3cm, xshift=2cm] at (phase1.north) {\textbf{Step 1: Eigenfunction Approximation}}; 

\node[draw=gray!50, dashed, inner sep=0.4cm, fit=(kernel)(approx)(eigen), label=above:\textbf{Step 2: Kernel Regression}] (phase2) {};
\node[draw=gray!50, dashed, inner sep=0.4cm, fit=(lyap)(cpa), label=above:\textbf{Step 3: Lyapunov Construction}] (phase3) {};

\end{tikzpicture}
}
 \caption{Schematic overview of the proposed kernel-based methodology for constructing Lyapunov functions using approximate Koopman eigenfunctions. The process involves three main steps: (1) eigenfunction approximation via PDE decomposition, (2) kernel regression in RKHS with error bounds (Theorem 3), and (3) Lyapunov construction and certification (Theorem 4).}
\label{fig:methodology_flowchart}
\end{figure}

The paper is organized as follows: Section~2 reviews Lyapunov functions and existing numerical approaches. Section~3 introduces the Koopman operator, its eigenfunctions, and establishes existence theory for the eigenfunction PDEs. Section~4 presents kernel methods and symmetric collocation. Section~5 details our main Lyapunov function construction and proves error bounds. Section~6 describes the CPA certification procedure. Section~7 demonstrates the method on numerical examples, and Section~8 concludes.


\section{Background on Lyapunov functions}\label{sec:dyn_sys}
We consider the $d$-dimensional dynamical system
\begin{equation}\label{odesys}
\Sigma: 
\begin{cases} 
\dot{x} = f(x), \\ 
x(0) = x_0,
\end{cases}
\end{equation}
defined on a state space $\cZ\subseteq \mR^d$, $d\in\mN$, with a vector field $f \in C^\sigma(\mathbb{R}^d, \mathbb{R}^d)$, $\sigma \geq 1$, and initial value 
$x_0\in\cZ$.
We denote by $s_t:\cZ\to\cZ$ the flow or solution map, i.e,. $s_t(x_0)=x(t)$ is the solution of~\eqref{odesys} at time $t\geq0$.

We assume that the origin $x=0$ is an hyperbolic equilibrium point for~\eqref{odesys}, i.e. $f(0) = 0$ and all eigenvalues of the Jacobian $E\coloneqq 
Df(0)\in\R^{d\times d}$ have a strictly negative real part. Recall that this implies in particular that $x=0$ is an exponentially asymptotically stable equilibrium 
(Definition 4.1 in~\cite{Khalil1996}).

Lyapunov functions are scalar-valued functions that enable the analysis of equilibrium stability in dynamical systems 
without explicitly solving the underlying equations. Named after Aleksandr Lyapunov, these functions are defined as follows.
\begin{definition}[Lyapunov function]
A function $V\in \cC^1(\mathbb{R}^d)$ is a Lyapunov function if it satisfies the following:
\begin{itemize}
    \item Positivity: \( V(x) > 0 \) for all \( x \neq 0 \), and \( V(0) = 0 \).
    \item Non-increasing behavior along trajectories of the system: \( \dot{V}(x) \leq 0 \), where 
    \begin{equation}\label{eq:decay_v}
\dot{V}(x) = \nabla V(x) \cdot f(x).    
    \end{equation}
\end{itemize}
\end{definition}
We recall that the existence of a Lyapunov function implies that the equilibrium $x=0$ is stable, and strict negativity $\dot{V}(x) < 0$ for $x\neq 0$ 
implies even asymptotic stability (Theorem 4.1 in~\cite{Khalil1996}).
Apart from studying the stability of equilibria, Lyapunov functions are used to estimate the basin of attraction $\mathcal{A} := \{x \in 
\mathbb{R}^d \mid s_t(x) \to 0 \text{ as } t \to \infty\}$, or to facilitate control design for feedback systems.

In addition to direct statements proving stability given the existence of a Lyapunov function, there are several partial converse results showing the 
existence of a Lyapunov function of suitable smoothness given certain stability properties of an equilibrium. 
We refer to Section 4.7 in~\cite{Khalil1996} for a detailed treatment, but we stress that these results are non-constructive, motivating computational methods.
Algorithms based on radial basis functions (RBFs) and partial differential equations (PDEs) are commonly used, see e.g.~\cite{lyap_bh, 
Giesl2007}. We recall their definition in the next section since they relate to our work, but we stress that we will pursue a different 
strategy based on the eigenvalues of the Koopman operator.

%

\subsection{Numerical approaches based on kernels approximation}

Giesl and Wendland~\cite{Giesl2007,Giesl2007b} proposed a numerical method to approximate Lyapunov functions using kernel methods, and specifically Radial 
Basis Functions (RBFs). 
This approach enforces~\eqref{eq:decay_v} by solving the PDE 
\begin{equation}\label{eq:giesl_pde}
L V(x) \coloneqq \nabla V(x) \cdot f(x)=-\|x\|^2,
\end{equation}
or even a with more general right-hand side $-p(x)$, where $p$ is positive and satisfying some conditions (see~\cite{Giesl2007} for more details).

An approximate solution $V^*$ of the PDE~\eqref{eq:giesl_pde} is found by symmetric kernel collocation (see Chapter 16 in~\cite{Wendland2005}, and 
Section~\ref{sec:sym_coll} below), typically using
a Wendland RBF kernel due to its compact support and smoothness properties~\cite{Wendland1995a,Wendland1998}.
Under certain smoothness conditions on $f(x)$ and on the kernel, which are completely analogous to the assumptions of our Theorem~\ref{th:error} below, the 
approximation error is bounded in~\cite{Giesl2007} as
\begin{equation*}
|\dot V(x) - \dot V^*(x)| \leq C \rho_{Z,\Omega}^\theta, \quad \forall x \in \Omega,
\end{equation*}
where 
\begin{equation*}
    \rho_{Z,\Omega}:=\sup_{x\in\Omega}\min_{z\in Z}\|x-z\|_2
\end{equation*}
is the fill distance of the finite set $Z\subset\Omega$ of collocation points, $\theta$ depends on the smoothness assumptions, and $\Omega$ is the bounded domain containing the origin where the Lyapunov function is constructed.

This approach has been extended to the case where $f(x)$ is unknown~\cite{lyap_bh}, where Giesl, Hamzi, Rasmussen, and Webster developed a framework for 
approximating Lyapunov functions from noisy observations of the system dynamics. The key idea is to use measured trajectories \(x(t)\) to approximate $f$ from 
data using kernels and then apply the algorithim of~\cite{Giesl2007}. This approach keeps the same advantages as the original approach, including 
error estimates.

\section{The Koopman operator and its spectrum} \label{sect:koopman}

In this section we provide a brief overview of the spectral theory of the Koopman operator,
and refer the reader to \cite{mezic2020spectrum,mezic2021koopman} for more details. 

Let $\cF\subseteq \cC^0(\R^d)$ be a function space of \emph{observables} $\psi: \cZ\to \mC$ for the system~\eqref{odesys}.
We have the following definitions for the Koopman operator and its spectrum (see also~\cite{Lasota}).
\begin{definition}[Koopman Operator] The family of Koopman  operators $\mathbb{U}_t:\cF\to \cF$ corresponding to~\eqref{odesys} is defined as 
\begin{eqnarray}\label{koopman_operator}
[\mathbb{U}_t \psi](x)=\psi(s_t(x)). 
\end{eqnarray}
If in addition $\psi$ is continuously differentiable, then $\varphi(x,t):=[\mU_t \psi ](x)$ satisfies the partial differential equation 
\begin{align}
\partial_t \varphi(x, t)=\nabla_x \varphi(x,t) \cdot f(x)\eqqcolon \cK_f \varphi(x, t), \label{Koopmanpde}
\end{align}
with the initial condition $\varphi(x,0)=\psi(x)$. The operator $\cK_f$ is the infinitesimal generator of $\mU_t$, i.e.,
\begin{eqnarray}
{\cal K}_{f} \psi=\lim_{t\to 0}\frac{(\mathbb{U}_t-I)\psi}{t}. \label{K_generator}
\end{eqnarray}
\end{definition}
It easily follows from~\eqref{koopman_operator} that each $\mU_t$ is a linear operator on the space of functions $\cF$.  

\begin{definition}[Eigenvalues and Eigenfunctions of the Koopman operator]\label{definition_koopmanspectrum}
A function $\varphi_\lambda\in\cF$, assumed to be at least $\cC^1(\cZ)$,  
is said to be an eigenfunction of the Koopman operator associated with eigenvalue $\lambda\in\mC$ if
\begin{eqnarray}\label{eig_koopman}
[\mU_t \varphi_\lambda](x)=e^{\lambda t}\varphi_\lambda(x),\;\;x\in\cZ.
\end{eqnarray}
Using the Koopman generator, equation~\eqref{eig_koopman} can be written as 
\begin{equation}\label{eig_koopmang}
    \nabla_x \varphi_\lambda(x)\cdot f(x)=\lambda \varphi_\lambda(x).
\end{equation}
\end{definition}
In the following discussion, we summarize the 
results from \cite{mezic2020spectrum} relevant to this paper and justify some of the claims made above on the spectrum of the Koopman operator. 
Equations \eqref{eig_koopman} and \eqref{eig_koopmang} provide a general definition of the Koopman spectrum. However, the spectrum can be defined over finite 
time or over a subset of the state space. The spectrum of interest to us in this paper could be well-defined over the subset of the state space. 
\begin{definition}[Open Eigenfunction \cite{mezic2020spectrum}]\label{definition_openeigenfunction}
Let $\varphi_\lambda: \bC\to \mC$, where $\bC\subset \cZ$ is not necessarily an invariant set. Let $x\in  \bC$, and
$\tau \in (\tau^-(x),\tau^+(x))= I_x$, a connected open interval such that $S_\tau (x) \in \bC$ for all  $\tau \in I_x$.
If
\begin{align}[\mU_\tau \varphi_\lambda](x) = \psi_\lambda(s_\tau(x)) =e^{\lambda \tau} \varphi_\lambda (x),\;\;\;\;\forall \tau \in I_x, 
\end{align}
then $\varphi_\lambda(x)$ is called the open eigenfunction of the Koopman operator family $\mU_t$, for $t\in \mR$ with eigenvalue $\lambda$. 
\end{definition}
If $\bC$ is a proper invariant subset of $\cZ$, in which case $I_x=\mR$ for every $x\in \bC$, then $\psi_\lambda$ is called the subdomain eigenfunction. If 
$\bC=\cZ$ then $\psi_\lambda$ will be the ordinary eigenfunction associated with eigenvalue $\lambda$ as defined in \eqref{eig_koopman}. The open 
eigenfunctions as defined above can be extended from $\bC$ to a larger reachable set when $\bC$ is open based on the construction procedure outlined in  
\cite[Definition 5.2, Lemma 5.1]{mezic2020spectrum}. 
Let $\cP$ be that larger domain. The eigenvalues of the linearization of the system dynamics at the origin, 
i.e., $E$, will form the eigenvalues of the Koopman operator (see \cite[Proposition 5.8]{mezic2020spectrum}, and the discussion at the end of this section). 
Our interest will be in constructing the corresponding 
eigenfunctions, defined over the domain $\cP$. We will refer to these eigenfunctions as {\it principal eigenfunctions} \cite{mezic2020spectrum}, which 
can be defined over a proper subset $\cP$ of the state space $\cZ$ (called subdomain eigenfunctions) or over the entire $\cZ$ 
\cite[Lemma 5.1, Corollary 5.1, 5.2, and 5.8]{mezic2020spectrum}.
The principal eigenfunctions can be used as a change of coordinates in the linear representation of a nonlinear system and draw a connection to the famous 
Hartman-Grobman theorem  on linearization and Poincare normal form~\cite{arnold2012geometrical}.

\subsection{Decomposition of Koopman Eigenfunctions} \label{sect_koopman_eigen_approx}




We consider for simplicity a principal eigenfunction $\varphi_\lambda$ with simple real eigenvalue $\lambda\in\R$. The extension 
to the complex case is deferred to future work. 

Following~\cite{deka2023path}, we decompose this eigenfunction into linear and nonlinear parts. 
First, we consider the linearization in $0$ of~\eqref{odesys}, namely
\begin{equation}\label{sys_decompose}
\dot x = f(x)=Ex+(f(x)-E x)\eqqcolon E x+ G(x),
\end{equation}
where $E$ is the Jacobian in $0$ as defined in Section~\ref{sec:dyn_sys}. It follows that $Ex$ is the linear part of $f$ at $0$ and $G(x)$ the purely 
nonlinear part. 

Similarly, $\varphi_\lambda$ admits a decomposition in $0$ into linear and nonlinear parts, that we may generically write as
\begin{equation}\label{eq:eigen}
 \varphi_\lambda(x)=w^\top x+h(x),
\end{equation}
where $w^\top\in\R^d$ and $h:\R^d\to\mR$ is a suitable nonlinear function. 
Substituting \eqref{eq:eigen} into~\eqref{eig_koopmang} gives thus
\begin{align*}
\nabla_x\left(w^\top x+h(x)\right)\cdot (E x+ G(x)) &= \lambda \left(w^\top x+h(x)\right)\\
\Rightarrow w^\top Ex + w^\top G(x) +\nabla_x h(x)\cdot f(x) &= \lambda w^\top x+\lambda h(x).
\end{align*}
Equating the linear and nonlinear parts on both sides (see~\cite{mezic2020spectrum} for the details), we obtain that $w$ and $h(x)$ need to satisfy
\begin{eqnarray}\label{linear_nonlinear_eig}
w^\top E=\lambda  w^\top,\;\;\nabla_x h(x) \cdot f(x)-\lambda h(x)+w^\top G(x)=0.
 \end{eqnarray}
So, the linear part of the eigenfunction can be found as the left eigenvector with eigenvalue $\lambda$ of the matrix $E$, and the nonlinear term satisfies the 
linear partial differential equation~\eqref{linear_nonlinear_eig}.
\subsection{Existence of Koopman Eigenfunctions}\label{sect:existence}

The decomposition~\eqref{eq:eigen} and the resulting PDE~\eqref{linear_nonlinear_eig} for the nonlinear part $h$ raise a fundamental question: under what conditions does a solution exist? This section establishes existence criteria that underpin the approximation theory developed in Section~\ref{sec:sym_coll} and the error bounds in Theorems~\ref{theorem_stability_validity} and~\ref{th:error}.

Recall from~\eqref{linear_nonlinear_eig} that the nonlinear part $h_{\lambda}$ of the eigenfunction $\varphi_\lambda$ must satisfy
\begin{equation}\label{eq:h_pde_exist}
    \nabla h_\lambda(x) \cdot f(x) - \lambda h_\lambda(x) = -w^\top G(x),
\end{equation}
where $w$ is the left eigenvector of $E = Df(0)$ corresponding to eigenvalue $\lambda$, and $G(x) = f(x) - Ex$ is the purely nonlinear part of the vector field. We impose the boundary conditions
\begin{equation}\label{eq:bc_exist}
    h_\lambda(0) = 0, \qquad \nabla h_\lambda(0) = 0,
\end{equation}
which ensure that $\varphi_\lambda(x) = w^\top x + h_\lambda(x)$ has the correct linear behavior near the equilibrium.

To analyze~\eqref{eq:h_pde_exist}, we define the first-order linear differential operator
\begin{equation}\label{eq:D_operator}
    \mathcal{D}_\lambda u := \nabla u \cdot f - \lambda u,
\end{equation}
so that~\eqref{eq:h_pde_exist} becomes $\mathcal{D}_\lambda h_\lambda = -w^\top G$. This operator $\mathcal{D}_\lambda$ coincides with the operator $D_i$ introduced in~\eqref{eq:h_pde} when $\lambda = \lambda_i$, and forms the basis of the kernel approximation framework in Section~\ref{sec:sym_coll}. The formal $L^2$-adjoint of $\mathcal{D}_\lambda$ is given by
\begin{equation}\label{eq:adjoint}
    \mathcal{D}_\lambda^* \psi := -\nabla \cdot (f\psi) - \lambda \psi.
\end{equation}

The following theorem provides a Fredholm-type criterion for the existence of solutions.

\begin{theorem}[Existence of Koopman Eigenfunctions]\label{thm:existence}
Let $\Omega \subset \mathbb{R}^d$ be an open bounded domain containing the origin, and let $f \in C^1(\Omega, \mathbb{R}^d)$ with $f(0) = 0$. Assume the origin is a hyperbolic equilibrium with all eigenvalues of $E = Df(0)$ having strictly negative real parts. Let $\lambda < 0$ be an eigenvalue of $E$ with corresponding left eigenvector $w \in \mathbb{R}^d$.

Then the PDE~\eqref{eq:h_pde_exist} with boundary conditions~\eqref{eq:bc_exist} has a solution $h_\lambda \in C^1(\Omega)$ if and only if the source term satisfies the orthogonality condition
\begin{equation}\label{eq:solvability}
    \int_\Omega w^\top G(x) \, \psi(x) \, dx = 0 \quad \text{for all } \psi \in \ker(\mathcal{D}_\lambda^*).
\end{equation}
\end{theorem}

\begin{proof}
This follows from the classical Fredholm alternative for first-order linear PDEs. The operator $\mathcal{D}_\lambda$ acts on functions in $C^1(\Omega)$ satisfying the boundary conditions~\eqref{eq:bc_exist}, and the equation $\mathcal{D}_\lambda h = g$ has a solution if and only if $g$ is $L^2(\Omega)$-orthogonal to the kernel of the adjoint $\mathcal{D}_\lambda^*$. Since the kernel of $\mathcal{D}_\lambda^*$ consists of functions $\psi$ satisfying $-\nabla \cdot (f\psi) - \lambda \psi = 0$, and for hyperbolic equilibria with $\lambda < 0$ this kernel is typically trivial or finite-dimensional, existence is guaranteed for generic source terms.
\end{proof}

\begin{remark}[Practical implications and path-integral formula]\label{rem:existence_practical}
For the systems considered in this paper, condition~\eqref{eq:solvability} is generically satisfied. When $\Omega$ is a neighborhood of the origin and $f$ is analytic, the eigenfunction $\varphi_\lambda$ can alternatively be constructed via the path-integral formula~\cite{deka2023path}:
\begin{equation}\label{eq:path_integral}
    \varphi_\lambda(x) = w^\top x + \int_0^\infty e^{-\lambda t} w^\top G(s_t(x)) \, dt,
\end{equation}
The above formula is valid  under the following assumption on the  eigenvalue
\[-{\rm Re}(\lambda)+2{\rm Re}(\lambda_{max})<0\]
where $\lambda_{max}$ is the eigenvalue with maximum real part which is less than zero as the equilibrium is asymptotically stable \cite{deka2023path}.
Here, $s_t(x)$ denotes the flow of~\eqref{odesys}. This formula provides an explicit representation when the integral converges, which is guaranteed for $\lambda < 0$ and trajectories in the basin of attraction. The path-integral representation also provides an alternative existence argument: for asymptotically stable equilibria, the integral converges absolutely for all $x$ in the basin of attraction, yielding a well-defined eigenfunction. Moreover, formula~\eqref{eq:path_integral} can be used to bound $\|\varphi_\lambda\|_{L^\infty(\Omega)}$ in terms of the system dynamics, as required in Section~5. We also note that the kernel approximation method of Section~\ref{sec:sym_coll} does not require verifying~\eqref{eq:solvability} directly; it finds the best approximation in the RKHS sense, which exists regardless of exact solvability.
\end{remark}

Theorem~\ref{thm:existence} ensures that the target function $h_\lambda$ we seek to approximate in Section~\ref{sec:sym_coll} is well-defined. The subsequent results build on this foundation: Theorem~\ref{theorem_stability_validity} establishes that solutions to~\eqref{eq:h_pde_exist}--\eqref{eq:bc_exist} possess Sobolev regularity $h_\lambda \in W^{m+1}_2(\Omega)$ when $f \in C^m$, providing the smoothness required for the error analysis. Theorem~\ref{th:error} then proves that the kernel collocation approximation $h^*_\lambda$ converges to $h_\lambda$ at a rate determined by the fill distance $\rho_{Z,\Omega}$ of the collocation points. Finally, the main result in Section~5 shows that the Lyapunov function $V^*(x) = \sum_{i,j} P_{ij} \varphi^*_i(x) \varphi^*_j(x)$ (for the definition of $P_{ij}$ see Section 5) constructed from approximate eigenfunctions satisfies $\|\dot V - \dot V^*\|_{L^\infty} = O(\rho_{Z,\Omega}^{m-d/2})$, with the error bound relying on the existence and regularity of each $\varphi_i$.

\section{Kernels and kernel-based symmetric collocation}


We give a brief overview of Reproducing Kernel Hilbert Spaces (RKHS) as used in statistical learning theory~\cite{CuckerandSmale} and approximation 
theory~\cite{Fasshauer2015,Wendland2005}. Early work developing
the theory of RKHS was undertaken by N. Aronszajn~\cite{aronszajn50reproducing}, and a recent comprehensive treatment can be found in~\cite{Saitoh2016}.

\begin{definition}[Reproducing Kernel Hilbert Spaces (RKHS)]\label{def:rkhs}
Let  ${\mathcal H}$  be a Hilbert space of real-valued functions on a nonempty set ${\mathcal X}$.
Denote by $\langle f, g \rangle_\cH$ the inner product on ${\mathcal H}$   and let $\|f\|_\cH= \langle f, f \rangle^{1/2}_\cH$
be the norm in ${\mathcal H}$, for $f$ and $g \in {\mathcal H}$. We say that ${\mathcal H}$ is a reproducing kernel
Hilbert space (RKHS) if there exists a function $k:{\mathcal X} \times {\mathcal X} \rightarrow \RR$
such that
\begin{itemize}
 \item[i.] $k_x:=k(x,\cdot)\in{\mathcal{H}}$ for all $x\in \cX$.
 \item[ii.] $k$ has the {\em reproducing property} $f(x)=\langle f,k_x \rangle_\cH$ $\forall f \in {\mathcal H}, x\in\mathcal X$.
\end{itemize}
The kernel $k$ will be called the reproducing kernel of ${\mathcal H}$, which will be sometimes denoted as ${\mathcal H}_k$.
\end{definition}

Using the two defining properties of Definition~\ref{def:rkhs}, we can readily see that reproducing kernels are automatically symmetric, since $k(x, y) = \inner{\cH}{k(\cdot, x), k(\cdot, y)} = k(y, x)$ by the symmetry of the inner product. For kernel approximation, we look for the following additional property.

\begin{definition}[Positive definite kernel]
A symmetric kernel $k:\cX\times\cX\to\R$ is positive definite if the associated kernel matrix $K_X\coloneqq (k(x_i, x_j))_{i,j=1}^n\in\R^{n\times n}$ is positive semidefinite for any set  $X\coloneqq \{x_1, \dots, x_n\}\subset\cX$ of  $n\in\N$ points.
\end{definition}
Given this definition, it is clear that any reproducing kernel is also positive definite by the positivity of the inner product in $\cH$. Thus, starting from a 
RKHS one always gets a symmetric positive definite kernel.
The opposite implication is also valid: Given a kernel $k$ and a set $\mathcal X$, there is a unique RKHS ${\mathcal H}_k$ having $k$ as its reproducing kernel.
In the approximation literature, this space is usually called the native space of $k$ on $\mathcal X$.

\subsection{Generalized interpolation and symmetric collocation of linear PDEs}\label{sec:sym_coll}
Positive definite kernels can be used to construct approximants which satisfy finitely many constraints provided by the evaluation of linear functionals on an 
unknown function $h\in \cH_k$.

In plain interpolation, the function $h$ is sampled at interpolation points $Z\coloneqq \{x_1,\dots,x_n\}\subset\cX$, and the kernel interpolant is defined as 
the function $h^*$ of minimal norm which satisfies $h^*(x_i)=h(x_i)$, i.e., as the solution of the quadratic problem
\begin{align}\label{eq:interp_problem}
h^*&\coloneqq\argmin\limits_{g\in\cH_k} \norm{\cH_k}{g}^2\\
& \text{s.t. } \delta_{x_i}(g)=\delta_{x_i}(h),\;\;1\leq i\leq n,\nonumber
\end{align}
where $\delta_{x_i}$ is the point-evaluation functional at $x_i\in Z$. The fact that $\delta_x\in\cH_k'$ for all $x\in\cX$ is equivalent to requiring that 
$\cH_k$ is a RKHS.
Since $h\in\cH_k$, by the Representer Theorem~\cite{Wahba1970,Schoelkopf2001v}, the solution of~\eqref{eq:interp_problem} can be expressed as
\begin{equation}\label{eq:kernel_interp}
h^*(x) = \sum_{j=1}^n \alpha_j k(x, x_j),
\end{equation}
where the coefficient vector $\alpha\coloneqq (\alpha_1, \dots, \alpha_n)^T\in\R^n$ is defined as
\begin{equation*}
\alpha = K_Z^{\dagger} y,
\end{equation*}
with $y\coloneqq (g(x_1),\dots, g(x_n))^T\in\R^n$ and $K_Z\coloneqq (k(x_i, x_j))_{i,j=1}^n\in\R^n$ the kernel matrix. To improve the stability of the 
computations, it is common to relax the constraints in~\eqref{eq:interp_problem} and solve 
instead the problem
\begin{equation*}
\min\limits_{g\in\cH_k} \left(\sum_{j=1}^n |g(x_j) - h(x_j)|^2 + \eta \norm{\cH_k}{g}^2\right),
\end{equation*}
with a regularization parameter $\eta>0$. In this case the solution has still the expression~\eqref{eq:kernel_interp}, but now the coefficient vector solves
\begin{equation*}
\alpha = (K_Z + \eta I)^{-1} y.
\end{equation*}

Problem~\eqref{eq:interp_problem} can be extended to include constraints represented by more general linear functionals, leading to generalized interpolation 
(Chapter 16 in~\cite{Wendland2005}). In particular, the approximation of the solution of linear PDEs can be framed in this context simply by requiring the PDE 
to be satisfied in the strong form in a set of interpolation (or collocation) points, leading to what is known as symmetric kernel 
collocation~\cite{Fasshauer1997}.

\subsection{Kernel approximation of Koopman eigenfunctions}\label{sec:kernel_approx}

We now apply the symmetric collocation framework to approximate the nonlinear part of Koopman eigenfunctions. For each eigenvalue $\lambda_i$, $i=1,\dots,d$, we define the linear differential operator 
\begin{equation}\label{eq:h_pde}
D_i u(x)\coloneqq \nabla_x u(x)\cdot f(x) - \lambda_i u(x),
\end{equation}
so that the PDE~\eqref{linear_nonlinear_eig} for the nonlinear part $h_{\lambda_i}$ becomes
\begin{equation}\label{eq:pde_Di}
D_i h_{\lambda_i}(x) = -w^\top G(x).
\end{equation}

The kernel collocation approximation $h_{\lambda_i}^*$ of $h_{\lambda_i}$ is defined as the solution of the constrained optimization problem
\begin{align}\label{eq:opt_prob}
h_{\lambda_i}^*\coloneqq\argmin_{g\in\cH_k}& \norm{\cH_k}{g}^2\\
\text{s.t. } g(0)&=0,\nonumber\\
\partial_{x^j} g(0)&=0,\quad 1\leq j\leq d,\nonumber\\
D_i g(x_j)&=-w^\top G(x_j),\quad 1\leq j\leq n,\nonumber
\end{align}
where $x\coloneqq (x^1, \dots, x^d)$ denotes the components of $x\in\R^d$, and $Z = \{x_1, \ldots, x_n\} \subset \Omega$ is the set of collocation points. The third constraint enforces the PDE~\eqref{linear_nonlinear_eig} at the collocation points, while the first two constraints ensure that $h_{\lambda_i}^*$ and $\nabla h_{\lambda_i}^*$ vanish at the origin, as required for the nonlinear part of the eigenfunction decomposition~\eqref{eq:eigen}.

By the Representer Theorem, the solution of~\eqref{eq:opt_prob} can be expressed as
\begin{equation}\label{eq:h_star}
h^*_{\lambda_i}(x) = \sum_{j=1}^n \alpha_j k_{\text{PDE}}(x, x_j) + \alpha_{n+1} k(x, 0) + \sum_{\ell=1}^{d} \alpha_{n+1+\ell} (\partial_{y^\ell} k(x, y))_{|y=0}, 
\end{equation}
where $k_{\text{PDE}}(x, x_j)$ is the Riesz representer of the functional $\delta_{x_j}\circ D_i$:
\begin{equation*}
k_{\text{PDE}}(x, x_j)\coloneqq (\delta_{x_j} \circ D_i)_y k(x, y) = (\nabla_y k(x, y))_{|y=x_j}\cdot f(x_j) - \lambda_i k(x, x_j).
\end{equation*}
The coefficients $\alpha_1, \ldots, \alpha_{n+1+d}$ are determined by solving the linear system arising from the constraints in~\eqref{eq:opt_prob}, possibly with a regularization parameter $\eta>0$ for numerical stability. We refer to~\cite{Wendland2005,Lee2025} for details. The approximation $h^*_{\lambda_i}(x)$ can be differentiated analytically by computing derivatives of the kernel translates.

We now state and prove the theoretical guarantees for this approximation. For the sake of completeness, we provide a full proof of the first result, while we refer to~\cite{Lee2025} for the proof of Theorem~\ref{th:error}.

\begin{theorem}[Regularity and stability]\label{theorem_stability_validity}
Let the equilibrium point $x_e = 0$ of the ODE~\eqref{odesys} be hyperbolic, and assume $f\in C^m(\Omega, \R^d)$ for some integer $m \geq 1$, where $\Omega$ is an open bounded neighborhood of the origin. Let $w \in \R^d$ be a left eigenvector of $E = Df(0)$ corresponding to eigenvalue $\lambda < 0$, and let $G(x) = f(x) - Ex$. Consider the PDE
\begin{equation}
D h(x) := \nabla h(x) \cdot f(x) - \lambda h(x) = -w^\top G(x), \quad x \in \Omega, \label{linear_nonlinear_eig2}
\end{equation}
with boundary conditions
\begin{equation}
h(0) = 0, \quad \nabla h(0) = 0. \label{boundary_conditions}
\end{equation}
Then:
\begin{enumerate}
\item The solution $h$ belongs to the Sobolev space $W^{m+1}_2(\Omega)$.
\item There exists a stability bound: for $p, q \in [1, \infty]$, there exist constants $C_D, C_0, C'_0 > 0$ such that
\begin{equation}
\| h \|_{L^p(\Omega)} \leq C_D \| D h \|_{L^q(\Omega)} + C_0 |h(0)| + C'_0 \| \nabla h(0) \|_{\ell^r}. \label{stability_bound}
\end{equation}
\end{enumerate}
\end{theorem}

\begin{proof}
\textbf{Part 1 (Sobolev regularity):} 
Since $f \in C^m(\Omega, \R^d)$ and $G(x) = f(x) - Ex$ with $G(0) = 0$ and $\nabla G(0) = 0$, we have $G \in C^m(\Omega, \R^d)$ with $G(x) = O(\|x\|^2)$ near the origin. The source term $-w^\top G(x)$ therefore belongs to $C^m(\Omega)$ and vanishes to second order at the origin.

The operator $D = \nabla(\cdot) \cdot f - \lambda(\cdot)$ is a first-order linear differential operator with $C^m$ coefficients. By the theory of linear PDEs with smooth coefficients (see, e.g.,~\cite{Evans2010}), if the source term $-w^\top G \in W^m_2(\Omega)$ and the boundary data satisfies~\eqref{boundary_conditions}, then the solution $h$ gains one derivative of regularity: $h \in W^{m+1}_2(\Omega)$.

More precisely, since $f(0) = 0$ and all eigenvalues of $Df(0) = E$ have negative real parts, the characteristic curves of the operator $D$ flow toward the origin. Combined with the boundary conditions~\eqref{boundary_conditions}, this ensures the solution is uniquely determined and inherits regularity from the source term. The gain of one derivative follows from elliptic-type estimates adapted to transport equations with damping (the $-\lambda h$ term with $\lambda < 0$ provides coercivity).

\textbf{Part 2 (Stability bound):}
To establish~\eqref{stability_bound}, we use the method of characteristics. Along a trajectory $\gamma(t) = s_t(x)$ of the flow, the function $h$ satisfies the ODE
\[
\frac{d}{dt}h(\gamma(t)) = \nabla h(\gamma(t)) \cdot f(\gamma(t)) = \lambda h(\gamma(t)) - w^\top G(\gamma(t)).
\]
Solving this linear ODE backward in time from a point $x \in \Omega$ to the origin (which is reached as $t \to \infty$ for trajectories in the basin of attraction), we obtain
\[
h(x) = \int_0^\infty e^{-\lambda t} w^\top G(s_t(x))\, dt,
\]
where we used $h(0) = 0$ and the fact that $s_t(x) \to 0$ as $t \to \infty$.

Taking norms and using H\"older's inequality:
\[
|h(x)| \leq \int_0^\infty e^{-\lambda t} |w^\top G(s_t(x))|\, dt \leq \|w\| \int_0^\infty e^{-\lambda t} \|G(s_t(x))\|\, dt.
\]
Since $\lambda < 0$, the exponential $e^{-\lambda t} = e^{|\lambda| t}$ grows, but for $x$ in a compact subset of the basin of attraction, the trajectory $s_t(x)$ converges to zero exponentially fast, and $\|G(s_t(x))\| = O(\|s_t(x)\|^2)$ decays faster than the exponential grows. This yields a finite bound depending on $\|Dh\|_{L^q}$ through the relation $Dh = -w^\top G$.

The terms involving $h(0)$ and $\nabla h(0)$ in~\eqref{stability_bound} account for perturbations of the boundary conditions; when~\eqref{boundary_conditions} holds exactly, these terms vanish.
\end{proof}

\begin{theorem}[Approximation error]\label{th:error}
Assume the hypotheses of Theorem~\ref{theorem_stability_validity} holds,  and in particular there are $m > d/2$, $p, q \in [1,\infty]$, $C_D > 0$, such that the solution of~\eqref{linear_nonlinear_eig2}--\eqref{boundary_conditions} satisfies $h_\lambda \in W^{m+1}_2(\Omega) \cap \cH_k$, and~\eqref{stability_bound} holds.

Assume that the RKHS satisfies $\cH_k \hookrightarrow W_2^{m+1}(\Omega)$, and let $h_\lambda^*$ be the solution of the optimization problem~\eqref{eq:opt_prob} with collocation points $Z \subset \Omega$ and no regularization.

Then there exist constants $\rho_0, C > 0$ depending on $d, m, \Omega, p, q$ (but not on $\lambda, f, Z, h_\lambda$) such that if the fill distance $\rho_{Z,\Omega} < \rho_0$, then
\begin{equation}\label{eq:error_bound}
\norm{L^p(\Omega)}{h_\lambda - h_\lambda^*} 
\leq C \left(\norm{W_2^m(\Omega,\R^d)}{f} + |\lambda|\right) \rho_{Z,\Omega}^{m - d\left(\frac12 -\frac1q\right)_+}\norm{\cH_k}{h_\lambda},
\end{equation}
where $(x)_+ := \max(x, 0)$.
\end{theorem}

To be more specific, we remark that the constant $C>0$ in~\eqref{eq:error_bound} depends only on the embedding $\cH_k \hookrightarrow W_2^{m+1}(\Omega)$, on the values of $d$ and $m$, and on the constant in the generic error bound of Theorem 12 in~\cite{narcowich2005sobolev} (see Theorem 5 in~\cite{Lee2025} for the explicit definition of $C$).

The final approximation of the Koopman eigenfunction is defined as
\begin{equation}\label{eq:eigen_approx}
\varphi_{\lambda}^*(x) := w^\top x + h_\lambda^*(x).
\end{equation}
Since the linear term $w^\top x$ is exact, the error bound of Theorem~\ref{th:error} applies directly to the eigenfunction approximation:
\[
\|\varphi_\lambda - \varphi_\lambda^*\|_{L^p(\Omega)} = \|h_\lambda - h_\lambda^*\|_{L^p(\Omega)} \leq C \left(\|f\|_{W^m_2} + |\lambda|\right) \rho_{Z,\Omega}^{m - d(1/2 - 1/q)_+} \|h_\lambda\|_{\cH_k}.
\]

\section{Constructing Lyapunov functions by approximated Koopman eigenfunctions}

For simplicity of notation, we write $\varphi_i\coloneqq \varphi_{\lambda_i}$, $1\leq i\leq d$, for the eigenfunctions associated to the Koopman eigenvalues 
$\lambda_1, \dots, \lambda_d$. Following~\cite{Lee2025b}, we define a Lyapunov function
\begin{equation*}
V(x) \coloneqq \sum_{i,j=1}^d P_{ij} \varphi_i(x) \varphi_j(x),
\end{equation*}
where $P\in\R^{d\times d}$ is the positive definite solution of the equation
\begin{equation*}
\Lambda^T P + P \Lambda < 0, 
\end{equation*}
with $\Lambda$ the diagonal matrix with $\lambda_1, \dots, \lambda_d$ on the diagonal.
This is indeed a Lyapunov function as shown in~\cite{Lee2025b}). 
In particular, equation~\eqref{eig_koopmang} implies that
\begin{equation}\label{eq:dt_eigen}
\frac{d}{dt}\varphi_i(x(t))
=\left(\nabla_x \varphi_i(x)\right)\cdot \dot x(t)
=\nabla_x \varphi_i(x)\cdot f(x)
= \lambda_i \varphi_i(x),
\end{equation}
and thus
\begin{align}\label{eq:dt_lyap}
\dot V(x) 
&= \sum_{i,j=1}^d P_{ij} \frac{d}{dt}\left(\varphi_i(x)  \varphi_j(x)\right)
= \sum_{i,j=1}^d P_{ij} \left[\frac{d}{dt}\left(\varphi_i(x)\right)\varphi_j(x)+\varphi_i(x) \frac{d}{dt}\varphi_j(x)\right]\\
&= \sum_{i,j=1}^d P_{ij} \left[\lambda_i \varphi_i(x)\varphi_j(x)+\varphi_i(x) \lambda_j \varphi_j(x)\right]
= \sum_{i,j=1}^d P_{ij}\left(\lambda_i + \lambda_j\right) \varphi_i(x)\varphi_j(x)\nonumber.
\end{align}
This Lyapunov function can be approximated by replacing the true $\varphi_i$ with their kernel approximation $\varphi_i^*$ defined in~\eqref{eq:eigen_approx}, 
and obtaining the approximation
\begin{equation}\label{eq:approx_lyap}
V^*(x) \coloneqq \sum_{i,j=1}^d P_{ij} \varphi_i^*(x) \varphi_j^*(x).
\end{equation}
This is clearly a positive definite function, and the following theorem proves a  bound like Theorem 2.1 in~\cite{lyap_bh} for this 
approximation.

\begin{theorem}
Assume that Theorem~\ref{th:error} (and thus Theorem~\ref{theorem_stability_validity}) holds with $p=\infty$ and a given $1\leq q\leq \infty$.
Then we have
\begin{align}\label{eq:bound_v_prime}
\norm{L_\infty(\Omega)}{\dot V - \dot V^*}
&\leq C_\lambda  (2 \bar \lambda + 1)  C_\varphi\rho_{Z,\Omega}^{m - d\left(\frac12-\frac1q\right)_+},\\
\norm{L_\infty(\Omega)}{V - V^*}
&\leq C_\lambda C_\varphi   \rho_{Z,\Omega}^{m - d\left(\frac12-\frac1q\right)_+}
,\nonumber
\end{align}
where $\Omega\subset\cZ$ is a neighborhood of the origin, and
\begin{align*}
C_\lambda&\coloneqq  4 d C \ \norm{F}{P} \left(\norm{W_2^m(\Omega,\RR^d)}{f} +\bar 
\lambda\right),\\
\bar\lambda&\coloneqq \max_{1\leq i\leq d}\lambda_i,\\
C_\varphi&\coloneqq \max_{1\leq i\leq d}\norm{L_\infty(\Omega)}{\varphi_i}\max_{1\leq i\leq d}\norm{\cH_k}{h_i}.
\end{align*}

\end{theorem}
\begin{proof}
Combining~\eqref{eq:eigen} and~\eqref{eq:eigen_approx}, we obtain
\begin{equation*}
\varphi^*_i(x) = \varphi_i(x) + \left(h_i^*(x) - h_i(x)\right),
\end{equation*}
and thus using~\eqref{eq:dt_eigen} gives
\begin{align*}
\frac{d}{dt} \varphi_i^*(x(t))
&= \frac{d}{dt} \varphi_i(x(t)) + \frac{d}{dt}\left(h_i^*(x(t)) - h_i(x(t)\right)\\
&= \lambda_i \varphi_i(x) + \nabla_x\left(h_i^*(x) - h_i(x)\right) \cdot f(x).
\end{align*}
The definition~\eqref{eq:h_pde} of the linear operator $D_i$ gives
\begin{equation*}
\nabla_x\left(h_i^*(x) - h_i(x)\right) \cdot f(x) 
= D_i \left(h_i^*(x) - h_i(x)\right) + \lambda_i \left(h_i^*(x) - h_i(x)\right).
\end{equation*}
Setting $S_i(x)\coloneqq h_i^*(x) - h_i(x)$, we thus obtain
\begin{align}\label{eq:phi_star_phi_prime_star}
\varphi_i^*(x) 
&= \varphi_i(x) + S_i(x),\\
\frac{d}{dt} \varphi_i^*(x)
&= \lambda_i \varphi_i(x) + \lambda_i S_i(x) + D_i (S_i)(x),\nonumber
\end{align}
and with these equalities, for all $i,j=1,\dots, d$ we can compute
\begin{align*}
\left(\frac{d}{dt}\varphi_i^*(x)\right)&\varphi^*_j(x)
=\left(\lambda_i \varphi_i(x) + \lambda_i S_i(x) + (D_i S_i)(x)\right)\left(\varphi_j(x) + S_j(x)\right)\\
&=
\lambda_i \varphi_i(x) \varphi_j(x)
+\lambda_i \varphi_i(x) S_j(x)
+ \left(\lambda_i S_i(x) + (D_i S_i)(x)\right)\left(\varphi_j(x)+ S_j(x)\right)
\\
&=\lambda_i \varphi_i(x) \varphi_j(x) + U_{ij}(x),
\end{align*}
where we defined $U_{ij}(x)\coloneqq \lambda_i \varphi_i(x) S_j(x)+ \left(\lambda_i S_i(x) + (D_i S_i)(x)\right)\left(\varphi_j(x)+ S_j(x)\right)$.

We can now compute the time derivative of~\eqref{eq:approx_lyap} as in~\eqref{eq:dt_lyap}. We have
\begin{align*}
\dot  V^*(x) 
&= \sum_{i,j=1}^d P_{ij} \left[\frac{d}{dt}\left( \varphi^*_i(x)\right)\varphi^*_j(x)+\varphi^*_i(x) \frac{d}{dt}\varphi^*_j(x)\right]\\
&= \sum_{i,j=1}^d P_{ij} \left[\lambda_i \varphi_i(x) \varphi_j(x) + \lambda_j \varphi_j(x) \varphi_i(x)\right]+\sum_{i,j=1}^n P_{ij} \left[U_{ij}(x) + 
U_{ji}(x)\right]\\
&= \dot V(x)+\sum_{i,j=1}^d P_{ij} \left(U(x) + U^T(x)\right)_{ij}\\
&= \dot V(x)+\inner{F}{P, U(x)+ U^T(x)},
\end{align*}
where $\inner{F}{}$ is the Frobenius inner product of two square matrices.

Let now $\Omega\subset \cX$ be a neighborhood of $0$. We have
\begin{equation*}
\norm{L_\infty(\Omega)}{\dot V - \dot V^*}
\leq \sup_{x\in \Omega} \left|\inner{F}{P, U(x)+ U^T(x)}\right|
\leq 2 \norm{F}{P}\cdot \sup_{x\in \Omega} \norm{F}{U(x)},
\end{equation*}
since $\inner{F}{A, B}\leq \norm{F}{A}\norm{F}{B}$ for any two matrices $A, B$, and $\norm{F}{A}=\norm{F}{A^T}$. It follows that
\begin{equation}\label{eq:tmp_V_bound}
\norm{L_\infty(\Omega)}{\dot V - \dot V^*}
\leq 2 \norm{F}{P} \sup_{x\in \Omega} \sqrt{\sum_{i,j=1}^d U_{ij}(x)^2}
\leq 2 d\ \norm{F}{P} \max_{i,j=1,\dots, d} \norm{L_\infty(\Omega)}{U_{ij}}.
\end{equation}

Now from Theorem~\ref{th:error}\footnote{We are actually using equation (26) in the proof~\cite{Lee2025} for the first bound, while the second one 
is the main statement of the theorem} we know that 
\begin{align*}
\norm{L_\infty}{D_i S_i} 
&=\norm{L_\infty}{D_i \left(h_i^*(x) - h_i(x)\right)} 
\leq C \cdot \varepsilon_i\\
\norm{L_\infty}{S_i} 
&= \norm{L_\infty}{h_i^*(x) - h_i(x)} 
\leq C \cdot \varepsilon_i,
\end{align*}
where 
\begin{equation*}
\varepsilon_i\coloneqq \left(\norm{W_2^m(\Omega,\RR^d)}{f} + \lambda_i\right) \rho_{Z,\Omega}^{m - d\left(\frac12 -\frac1q\right)_+}\norm{\cH_k}{h_i}.
\end{equation*}
Setting $\bar\lambda\coloneqq \max_{1\leq i\leq d}\lambda_i$ and 
\begin{equation*}
\varepsilon 
\coloneqq \max_{1\leq i\leq d} \varepsilon_i
\leq \left(\norm{W_2^m(\Omega,\RR^d)}{f} + \bar \lambda\right) \rho_{Z,\Omega}^{m - d\left(\frac12 -\frac1q\right)_+}\max_{1\leq i\leq 
d}\norm{\cH_k}{h_i},
\end{equation*}
we get the bound
\begin{align*}
\norm{L_\infty(\Omega)}{U_{ij}}
&= \norm{L_\infty(\Omega)}{\lambda_i \varphi_i S_j+ \left(\lambda_i S_i + D_i (S_i)\right)\left(\varphi_j+ S_j\right)}\\
&= \norm{L_\infty(\Omega)}{\lambda_i \varphi_i S_j+ \lambda_i S_i\varphi_j + \lambda_i S_iS_j + D_i (S_i)\varphi_j + D_i (S_i)S_j}\\
&\leq \lambda_i\norm{L_\infty}{\varphi_i}C \varepsilon_j 
+ |\lambda_i\norm{L_\infty}{\varphi_j}C\varepsilon_i + \lambda_iC^2 \varepsilon_i\varepsilon_j + C\varepsilon_i 
\norm{L_\infty}{\varphi_j} + C^2 \varepsilon_i\varepsilon_j\\
&\leq C \varepsilon\cdot \left(2 \bar\lambda\max_{1\leq i\leq d}\norm{L_\infty(\Omega)}{\varphi_i} + C\bar\lambda \varepsilon + \max_{1\leq i\leq 
d}\norm{L_\infty(\Omega)}{\varphi_i} + C\varepsilon\right)\\
&\leq C \varepsilon\cdot \left((2 \bar\lambda+1) \max_{1\leq i\leq d}\norm{L_\infty(\Omega)}{\varphi_i} + C(\bar\lambda+1) \varepsilon\right)\\
&\leq 2 C \varepsilon (2 \bar\lambda + 1)\max_{1\leq i\leq d}\norm{L_\infty(\Omega)}{\varphi_i},
\end{align*}
where the last step assumes that the fill distance $\rho_{Z,\Omega} < \rho_0$ is sufficiently small and thus $\varepsilon$ is also small enough.
Inserting this bound in~\eqref{eq:tmp_V_bound} proves the first inequality in~\eqref{eq:bound_v_prime}.

Proceeding similarly and using~\eqref{eq:phi_star_phi_prime_star}, we may write
\begin{equation*}
V(x) = \sum_{i,j=1}^d P_{ij} \varphi_i(x)^*\varphi_j(x)^* = V(x) + \sum_{i,j=1}^d P_{ij} W_{ij}(x),
\end{equation*}
where now
\begin{equation*}
W_{ij}(x) \coloneqq S_i(x) \varphi_j(x)+\varphi_i(x) S_j(x) +S_i(x) S_j(x),
\end{equation*}
and thus again we get
\begin{equation*}
\norm{L_\infty(\Omega)}{V - V^*}
\leq \sup_{x\in \Omega} \left|\inner{F}{P, W(x)}\right|
\leq d\ \norm{F}{P} \max_{i,j=1,\dots, d} \norm{L_\infty(\Omega)}{W_{ij}}.
\end{equation*}
Now $W_{ij}$ can be bounded as $U_{ij}$, obtaining  
\begin{align*}
\norm{L_\infty(\Omega)}{W_{ij}}
&= \norm{L_\infty(\Omega)}{\varphi_i S_j+ \varphi_i S_j + S_i S_j}
\leq C \varepsilon\left(2 \max\limits_{1\leq i\leq d}\norm{L_\infty(\Omega)}{\varphi_i} + C \varepsilon\right)\\
&\leq 3 C \varepsilon  \max\limits_{1\leq i\leq d}\norm{L_\infty(\Omega)}{\varphi_i},
\end{align*}
again by assuming that $\varepsilon$ is small enough. This concludes the proof.
\end{proof}

\begin{remark}[Explicit bounds on the constants]\label{rem:explicit_bounds}
Some constants appearing in the error bound can be further estimated in terms of the system parameters. 

For the matrix $P$, which solves the Lyapunov equation $\Lambda^\top P + P\Lambda = -I$, the integral representation $P = \int_0^\infty e^{\Lambda^\top t} e^{\Lambda t} \, dt$ yields the bound
\begin{equation}\label{eq:P_bound}
\|P\|_F \leq \frac{M^2}{2\alpha},
\end{equation}
where $\alpha := \min_{1\leq i\leq d} |\mathrm{Re}(\lambda_i)| > 0$ is the spectral abscissa and $M > 0$ is the constant in the matrix exponential bound $\|e^{\Lambda t}\| \leq M e^{-\alpha t}$.

For the eigenfunction norms, the path-integral formula~\eqref{eq:path_integral} gives
\[
\|\varphi_i\|_{L^\infty(\Omega)} = \sup_{x \in \Omega} \left| w_i^\top x + \int_0^\infty e^{-\lambda_i t} w_i^\top G(s_t(x)) \, dt \right|,
\]
which can be computed numerically by simulating trajectories and applying quadrature. 
Except for the term $\max_{1\leq i\leq d}\|h_i\|_{\cH_k}$, the constants $C_\lambda$ and $C_\varphi$ in the main estimate are thus bounded by quantities depending only on the linearization spectrum (through $\Lambda$), the nonlinearity (through $G$), and the domain geometry (through $\Omega$).
\end{remark}

\section{Construction of certified Lyapunov functions}\label{sec:certification}
For this section we additionally assume that $f\in C^2$.
The CPA method attempts to parameterize a continuous, piecewise affine Lyapunov function for a nonlinear system.  The essential idea is to enforce $V(0)=0$, $V(x_i)\ge a\|x_i\|_2$, and $\nabla V(x_i)f(x_i) \le -b\|x\|_2$ for some $a,b>0$, at the vertices at a given triangulation and then interpolate $V$, using the values at the vertices, 
over the simplices of the triangulation.  Indeed, this was done in a preliminary version of the CPA method in \cite{JGD1999cpajulian}. However, if we merely know that $\nabla V(x_i)f(x_i) \le -b\|x_i\|_2$ at the vertices $x_i$, then we might have $\nabla V(x)f(x) > 0$ for some $x$ in the interior of some simplex and $V$ is not a Lyapunov function.
In \cite{JGD1999cpajulian} it was checked a\,posteriori if the interpolated function truly was a Lyapunov function for the system at hand. 
In \cite{Mar2002cpa} this problem was solved in a constructive manner by adding the term $E_{\nu,i}\|\nabla V_\nu\|_1$ to the left-hand-side of $\nabla V(x_i)f(x_i) \le -b\|x\|_2$, where $E_{\nu,i}$ is positive constant that depends on the system and the simplex being checked, and of which $x_i$ is a vertex.  That is, the 
constraints to assert the negativity of the orbital derivative $\nabla V(x)\cdot f(x)$ on the simplex $\fS_\nu$ are:
\begin{equation}
\nabla V(x_i)\cdot f(x_i) + \|\nabla V_\nu\|_1 E_{\nu,i} \le -b\|x_i\|_2
\end{equation}
for every vertex $x_i$, $i=0,1,\ldots,d$, of the simplex $\fS_\nu$.
Note that $\nabla V_\nu$ is the  constant gradient of $V$ on the simplex $\fS_\nu$ and $\|\nabla V_\nu\|_1 \le C$ can be modelled through constraints that are linear in the values $V(x_i)$.
   The constants $E_{\nu,i}$, $\nu=1,2,\ldots,n$, $i=0,1,2,\ldots,d$, are defined as
   \begin{equation}
   E_{\nu,i} := \frac{1}{2}\sum_{r,s=1}^d B^\nu_{rs}|[x_i^\nu-x_0^\nu]_r|\big(|[x_i^\nu-x_0^\nu]_s|+ |[x_i^\nu-x_0^\nu]_d| \big), 
   \end{equation}
where $[x]_j$ is the $j$th component of the vector $x$ and 
\begin{equation}
B_{rs}^\nu\ge \max_{\substack{j=1,2,\ldots,d \\ x\in \fS_\nu}}\left|\frac{\partial^2f_j}{\partial x_r \partial x_s}(x)\right|.
\end{equation}
Let us recapitulate the CPA for completeness:  Given is the system~\eqref{odesys}, where we assume $\sigma\geq 2$, and a triangulation $\cT=(\fS_\nu)_{\nu=1,2,\ldots,d}$ of a compact neighbourhood $\Omega\subset \R^d$ of the origin.
The triangulation consist of $d$-simplices $\fS_\nu  =   \co\{x_0^\nu, \dots, x_d^\nu\} $, such that two different simplices intersect in a common face or not at all. We denote the set of all vertices
of the simplices in $\cT$ by $\cV$, and  we demand that the origin is in $\cV$, i.e.~$0\in\cV$.
Given values $V_{x}\in\R$, $x\in\cV$, we can define a continuous, piecewise affine function $V:\Omega\to \R$ by defining on each  $\fS_\nu  =   \co\{x_0^\nu, \dots, x_n^\nu\} $
\begin{equation}
\label{Vinterp}
V(x)=\sum_{i=0}^d\lambda_i(x)V_{x_i^\nu},
\end{equation}
 where the $\lambda_i:\fS_\nu\to [0,1]$ are the normalized barycentric coordinates of $x$ in $\fS_\nu$; that is $x=\sum_{i=0}^d\lambda_i(x)x_i^\nu$ and $\sum_{i=0}^d\lambda_i(x)=1$.

 If  the values $V_{x}\in\R$, $x\in\cV$, additionally fulfill, for some constants $a,b>0$, the constraints
 \begin{equation}
 \label{LC1}
 V_0=0\ \ \text{and}\ \ V_x\ge a\|x\|_2\ \ \text{for all}\  x\in\cV
 \end{equation}
 and, for every simplex $\fS_\nu  =   \co\{x_0^\nu, x_1^\nu,\dots, x_d^\nu\} $ and every $i=0,1,\ldots,d$,
 \begin{equation}
 \label{LC2}
 \nabla V(x_i^\nu)\cdot f(x_i^\nu) + \|\nabla V_\nu\|_1 E_{\nu,i} \le -b\|x_i^\nu\|_2
 \end{equation}
 then $V$ defined through \eqref{Vinterp} is a Lyapunov function for the system \eqref{odesys} that fulfills $V(0)=0$, $V(x)\ge a\|x\|_2$ for all $x\in \Omega$, and
 \begin{equation}
 \limsup_{h\to 0+} \frac{V(x+hf(x))-V(x)}{h} \le -b\|x\|_2
 \end{equation} 
for all $x\in\Omega^\circ$.  Hence, the origin is exponentially stable and every compact sublevel-set $\{x\in \Omega\colon V(x)\le c\}$, $c>0$, is in its basin of attraction.

Now, the constraints  \eqref{LC1} and \eqref{LC2} can be stated as the constraints of a linear programming  problem with variables $V_x$, $x\in\cV$, together with some auxiliary variables to model $\|\nabla V_\nu\|_1$.   A feasible solution to this problem allows one to parameterize a Lyapunov function $V$ for the system on $\Omega$ as in  \eqref{Vinterp}.
Another possibility, and the one we will follow here, is to assign values to the variables $V_x= W(x)$, $x\in \cV$, where $W$ is the function from the Koopman regression in Section \ref{sec:certification}, and then we verify the constraints \eqref{LC1} and \eqref{LC2} for these values. Note that since 
$\Omega$ is compact this boils down to verifying 
 \begin{equation}
 \label{LC1x}
 V_0=0\ \ \text{and}\ \ V_x>0\ \ \text{for all}\  x\in\cV
 \end{equation}
 and, for every simplex $\fS_\nu  =   \co\{x_0^\nu, x_1^\nu,\dots, x_d^\nu\} $ and every $i=0,1,\ldots,d$,
 \begin{equation}
 \label{LC2x}
 \nabla V(x_i^\nu)\cdot f(x_i^\nu) + \|\nabla V_\nu\|_1 E_{\nu,i} <0,
 \end{equation}
with the exception that for $x_i^\nu=0$ nothing must be checked if $x_0^\nu=0$ for every simplex $\fS_\nu$ such that $0\in\fS_\nu$ as \eqref{LC2} is trivially fulfilled with both sides equal to zero.
This verification is orders of magnitude faster than solving the linear programming problem.  A similar approach has been followed for Lyapunov functions constructed using generalized interpolation in RKHS \cite{GiHa2015combi} or the Kurzweil \cite{kurzweil1963} / Massera \cite{Massera1956} integral formula, see e.g.~\cite{HaVa2019intLya}.

\section{Examples}

\subsection{Example 1}
Consider the following nonlinear dynamical system:

\begin{equation}\label{eq:ex_one}
\dot x = 
 \left[ \begin{array}{c}
 -2x_1
 \\  -3 (x_2-x_1^2)
\end{array}\right].
\end{equation}

The true eigenfunctions and eigenvalues of the Koopman operator are

\begin{gather*}
\varphi_{\lambda_1}(x)= x_1,\quad  \lambda_1 = -2 \text{ and}\\
\varphi_{\lambda_2}(x)= x_2 +3x_1^2,\quad \lambda_2 = -3.
\end{gather*}
We first learn the eigenfunctions $\varphi_{\lambda_1}$ and $\varphi_{\lambda_2}$ across 3600 collocation points in the square grid $[-5,5] \times [-5,5]$ using the kernel collocation procedure described in Section~\ref{sec:kernel_approx} with the Gaussian kernel

\begin{equation}
    K(x,x') = \exp \Big(-\frac{\|x-x'\|^2}{2\sigma^2}\Big),
\end{equation}
where $\sigma = 3$. We then construct the Lyapunov function $V(x)$ following the methodology of Section~5. Figure~\ref{V} shows that our $V^*$ recovered with kernels accurately approximates the true $V$ over this larger domain.

For the CPA certification below, we restrict to the smaller domain $[-2,2]^2$ with a finer grid, demonstrating that a Lyapunov function learned on a coarser grid over a larger domain can be successfully certified on a refined mesh---this illustrates the generalization capability of the kernel-based approach.

Finally, we compute $\dot{V}$ using the chain rule formula from Section~5 and compare our results with the true $\dot{V}^*$ in Figure~\ref{V_dot}. Once again, $\dot{V}^*$ matches the true $\dot{V}$ in the graphs.

\begin{figure}[tbh]
\centering
\subfloat{\label{V_approx}
\includegraphics[width=.45\textwidth]{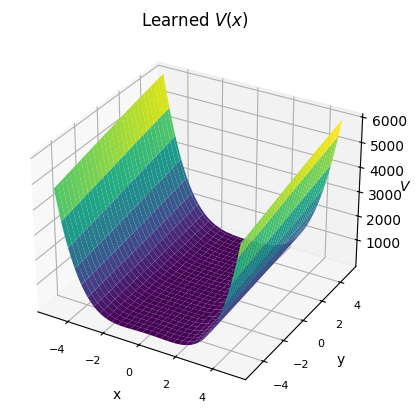}
}
\hfill
\subfloat{\label{V_true}
\includegraphics[width=.45\textwidth]{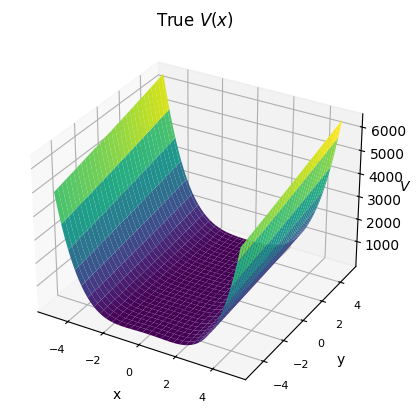}
}
\hfill
 \caption{Example \eqref{eq:ex_one}: approximated ${V}^*$ (left) and true ${V}$ (right), using 3600 collocation points}
\label{V}
\end{figure}
 
\begin{figure}[tbh]
\centering
\subfloat{\label{V_dot_approx}
\includegraphics[width=.45\textwidth]{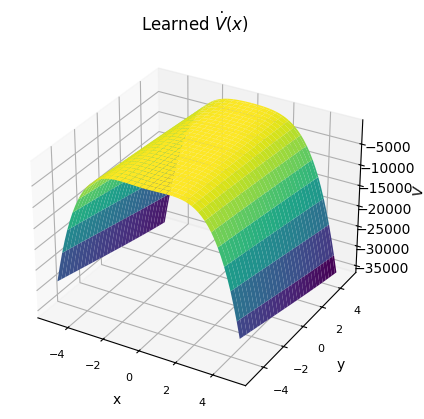}
}
\hfill
\subfloat{\label{V_dot_true}
\includegraphics[width=.45\textwidth]{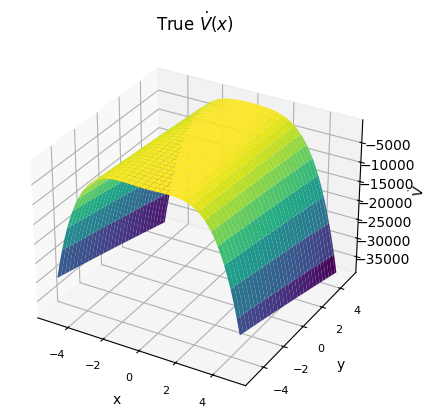}
}
\hfill
 \caption{Example \eqref{eq:ex_one}: approximated $\dot{V}^*$ (left) and true $\dot{V}$ (right), using 3600 collocation points}
\label{V_dot}
\end{figure}

\subsubsection*{Certification}
 
We used the Koopman regression to compute a Lyapunov function for the system, evaluating it on a regular grid of $109 \times 109=11{,}881$ points on $[-2,2]^2$.

For the CPA certification we used a regular triangulation of the same area consisting of $108\times 108\times 2=23{,}328$ congruent triangles.  It is easy to see that we can set $B^\nu_{11}=6$ and $B^\nu_{12}=B^\nu_{21}=B^\nu_{22}=0$.
The results can be seen in Figure~\ref{fig:Lya}.  The orbital derivative is negative with the exception of a very small area close to the origin (red dots in the figure).  This is to be expected unless special care is taken for the triangulation at the equilibrium, see \cite{Giesl2010,GiHa2012excpandim}.
A more efficient method to assert stability and obtain a large inner approximation of the basin of attraction is to use a local quadratic Lyapunov function for the linearization of the system close to the equilibrium to assert local stability and obtain a small ball $B$ around the equilibrium that is guaranteed to be in the basin of attraction and then to use CPA certification to obtain a larger estimate.  If the orbital derivative is certified to be negative on $\{x\in\R^n\colon V(x)\le c\}\setminus B$, where $\{x\in\R^n\colon V(x)\le c\}$ is a compact subset of $\Omega$, then every solution starting in the sublevel set will end up in $B$ and thus converge to the equilibrium.
     
\begin{figure}[H]
\centering
\includegraphics[width=0.5\textwidth]{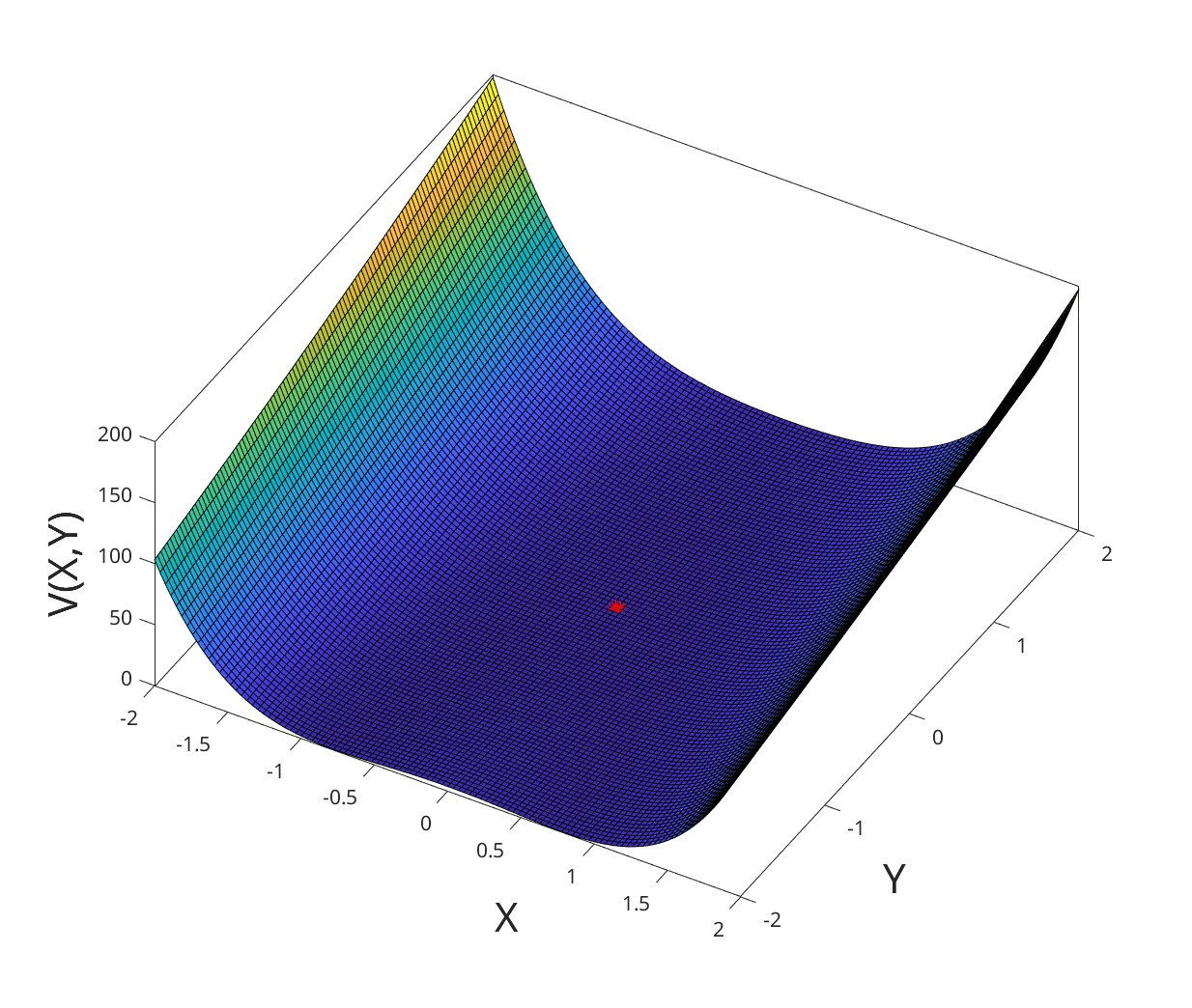}
\caption{Lyapunov function for the system~\eqref{eq:ex_one} computed by Koopman regression and certified with the method in Section~\ref{sec:certification}.  The orbital derivative is negative with exception of the red dots close to the origin.}
\label{fig:Lya}
\end{figure}

\subsection{Example 2}

Consider the Duffing oscillator with damping:
\begin{equation}
\begin{aligned}
\dot{x}_1 &= x_2, \\
\dot{x}_2 &= -\delta x_2 - \alpha x_1 - \beta x_1^3,
\end{aligned}
\label{eq:duffing}
\end{equation}
where $\delta > 0$ is the damping coefficient, and $\alpha, \beta > 0$ define the stiffness and nonlinearity, respectively.

The linearization at the origin has Jacobian
\begin{equation*}
E = \begin{pmatrix} 0 & 1 \\ -\alpha & -\delta \end{pmatrix},
\end{equation*}
with characteristic polynomial $\lambda^2 + \delta\lambda + \alpha = 0$ and eigenvalues
\begin{equation*}
\lambda = \frac{-\delta \pm \sqrt{\delta^2 - 4\alpha}}{2}.
\end{equation*}
For $\delta > 0$ and $\alpha > 0$, both eigenvalues have negative real parts: either two negative real eigenvalues when $\delta^2 \geq 4\alpha$ (overdamped case), or complex conjugate eigenvalues with negative real part $-\delta/2$ when $\delta^2 < 4\alpha$ (underdamped case). Thus, the origin is an asymptotically stable equilibrium.

Define the following Lyapunov function, corresponding to the total mechanical energy of the system:
\begin{equation}
V(x) = \frac{1}{2}x_2^2 + \frac{1}{2}\alpha x_1^2 + \frac{1}{4}\beta x_1^4.
\label{eq:lyapunov}
\end{equation}

This function is:
\begin{itemize}
    \item Positive definite: $V(x) > 0$ for all $x \ne 0$, and $V(0) = 0$.
    \item Radially unbounded: $V(x) \to \infty$ as $\|x\| \to \infty$.
\end{itemize}

The derivative of $V$ along system trajectories is:
\begin{equation}
\begin{aligned}
\dot{V}(x) &= \frac{\partial V}{\partial x_1} \dot{x}_1 + \frac{\partial V}{\partial x_2} \dot{x}_2 \\
&= (\alpha x_1 + \beta x_1^3)x_2 + x_2(-\delta x_2 - \alpha x_1 - \beta x_1^3) \\
&= -\delta x_2^2.
\end{aligned}
\end{equation}

This shows that $\dot{V}(x) \le 0$ for all $x$, with equality only when $x_2 = 0$. By Lyapunov's direct method and LaSalle's invariance principle:
\begin{enumerate}
    \item[(i)] The origin is a globally asymptotically stable equilibrium point.
    \item[(ii)] Trajectories converge to the largest invariant set in $\{x : x_2 = 0\}$, which is the origin since $\alpha, \beta > 0$.
\end{enumerate}
Note, however, that this function $V$ is not a {\it strict} Lyapunov function. 

Using the method proposed in this paper, we now seek to compute a strict Lyapunov function. We use the same Gaussian kernel with parameter $\sigma=3$ and 3600 points in the same domain $[-5,5]^2$ as Example 1 to construct the Koopman eigenfunctions and then the Lyapunov functions.
We still plot the graphs of the (nonstrict) analytical Lyapunov function above next to the computed Lyapunov function $V^*$, but in this case $V^*$ is not approximating $V$ and thus the functions are different. In contrast to $V$, The computed Lyapunov function $V^*$ is a strict Lyapunov function, satistying 
positivity and negativity respectively on the domain as well as convexity.

\begin{figure}[tbh]
\centering
\subfloat{\label{V_approx_duffing}
\includegraphics[width=.45\textwidth]{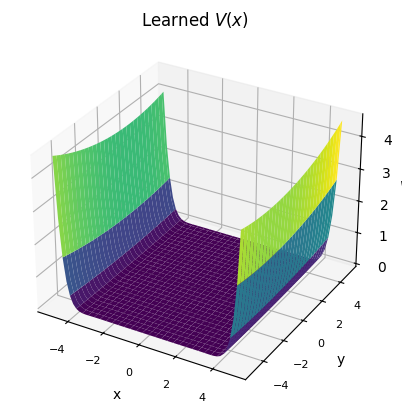}
}
\hfill
\subfloat{\label{V_true_duffing}
\includegraphics[width=.45\textwidth]{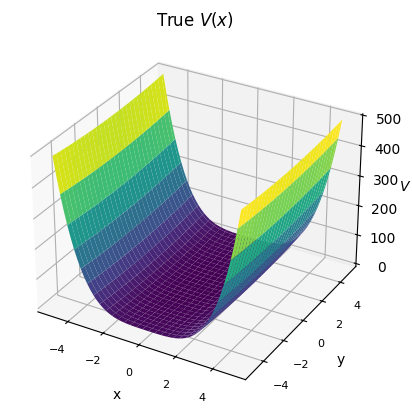}
}
\hfill
 \caption{Example \eqref{eq:duffing}: approximated ${V}^*$ (left) and true ${V}$ (right), using 3600 collocation points  -- note that $V^*$ is not approximating $V$, in fact, $V^*$ is a strict Lyapunov function while $V$ is not}
\label{V2}
\end{figure}

\begin{figure}[tbh]
\centering
\subfloat{\label{V_dot_approx_duffing}
\includegraphics[width=.45\textwidth]{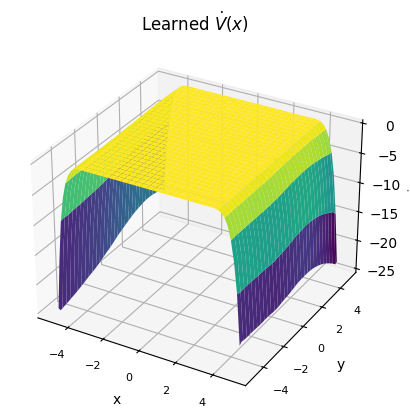}
}
\hfill
\subfloat{\label{V_dot_true_duffing}
\includegraphics[width=.45\textwidth]{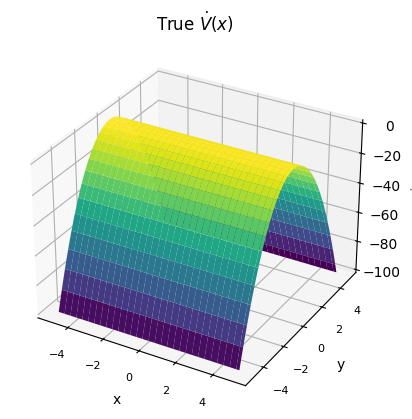}
}
\hfill
 \caption{{Example \eqref{eq:duffing}: approximated $\dot{V}^*$ (left) and true $\dot{V}$ (right), using 3600 collocation points -- note that $V^*$ is not approximating $V$, in fact, $V^*$ is a strict Lyapunov function while $V$ is not}}
\label{V_dot2}
\end{figure}

\section{Conclusion}

We have presented a kernel-based methodology for constructing Lyapunov 
functions using approximate Koopman eigenfunctions. Our approach leverages 
the decomposition of principal eigenfunctions into linear and nonlinear 
components, where the nonlinear part is computed via symmetric kernel 
collocation in RKHS. The resulting Lyapunov function, constructed as a 
quadratic form in the eigenfunctions, inherits rigorous error bounds 
(Theorem 4) that depend on the fill distance of collocation points.

The CPA certification procedure provides a mechanism to verify that the 
computed function is indeed a valid Lyapunov function outside the training 
points, addressing a key practical concern in data-driven stability analysis.

Among the limitations of the current framework is that it is restricted to systems with real 
Koopman eigenvalues. Extension to complex eigenvalues, which arise in 
oscillatory systems, requires considering conjugate pairs and is deferred 
to future work. Additionally, the method requires knowledge of the vector 
field $f(x)$; a fully data-driven extension using trajectory measurements 
is an important direction.

 Several extensions are under investigation: i.) Learning optimal kernels to maximize the certified domain of attraction, ii.)
Purely data-driven approaches that learn local Lyapunov functions 
   without explicit knowledge of $f$, iii.) Improved verification tools leveraging interval arithmetic or 
   sum-of-squares certificates

The combination of Koopman spectral theory and kernel methods provides a 
promising framework for stability analysis that balances theoretical 
guarantees with computational tractability.

\section*{Acknowledgement}
GS is a member of INdAM-GNCS, and his work was partially supported by the project ``Data-driven discovery and control of multi-scale interacting artificial agent system'' funded by the program Next-GenerationEU - National Recovery and Resilience Plan (NRRP) -- CUP H53D23008920001.
HO and JL acknowledge support from the DoD Vannevar Bush Faculty Fellowship Program under ONR award number N00014-18-1-2363 and the Air Force Office of Scientific Research under MURI award number FOA-AFRL-AFOSR-2023-0004 (Mathematics of Digital Twins). UV will acknowledge support from NSF CMMI grant 2031573.

\bibliographystyle{abbrv}
\bibliography{bibliography}

\end{document}